\documentclass[a4paper,11pt,onecolumn]{amsart}
\usepackage{amsfonts}
\usepackage{pb-diagram}
\title[Residual properties of $3$-manifold groups]{Residual properties of $3$-manifold groups I: Fibered and hyperbolic $3$-manifolds}

\newtheorem{thm}{Theorem}
\newtheorem{lemma}[thm]{Lemma}

\newtheorem{cor}[thm]{Corollary}
\newtheorem{prop}[thm]{Proposition}

\newtheorem{quest}[thm]{Question}

\numberwithin{thm}{section}

\newcommand{\al}{\alpha}

\newcommand{\gam}{\Gamma}

\newcommand{\bZ}{\mathbb{Z}}
\newcommand{\bR}{\mathbb{R}}
\newcommand{\bC}{\mathbb{C}}

\DeclareMathOperator{\Hom}{Hom}
\DeclareMathOperator{\Aut}{Aut}

\DeclareMathOperator{\Mod}{Mod}
\DeclareMathOperator{\End}{End}

\newcommand{\cL}{\mathcal{L}}

\newcommand{\bQ}{\mathbb{Q}}

\newcommand{\bH}{\mathbb{H}}

\newcommand{\whG}{\widehat{G}}

\newcommand{\whO}{\widehat{\mathcal{O}}}

\author[T. Koberda]{Thomas Koberda}
\address{Department of Mathematics\\ Harvard University\\ 1 Oxford St.\\ Cambridge, MA 02138 }
\email{ koberda@math.harvard.edu}
\begin{document}
\begin{abstract}
Let $p$ be a prime.  In this paper, we classify the geometric $3$-manifolds whose fundamental groups are virtually residually $p$.  Let $M=M^3$ be a virtually fibered $3$-manifold.  It is well-known that $G=\pi_1(M)$ is residually solvable and even residually finite solvable.  We prove that $G$ is always virtually residually $p$.  Using recent work of Wise, we prove that every hyperbolic $3$-manifold is either closed or virtually fibered and hence has a virtually residually $p$ fundamental group.  We give some generalizations to pro-$p$ completions of groups, mapping class groups, residually torsion-free nilpotent $3$-manifold groups and central extensions of residually $p$ groups.
\end{abstract}
\maketitle
\begin{center}
\today
\end{center}
\tableofcontents
\section{Introduction and statement of results}
Let $M$ be a compact orientable $3$-manifold with toral boundary.  The following is a recent result of Perelman on a conjecture of Thurston (cf. \cite{T2}, \cite{P1}, \cite{P2}, \cite{P3}):
\begin{prop}[Geometrization theorem, cf. Thurston and Perelman]
Every oriented irreducible closed $3$-manifold with toral boundary can be cut along a finite collection of incompressible tori so that the resulting pieces are ``geometric", in the sense that they admit finite volume geometric structures modeled on one of the eight model geometries.
\end{prop}

Recall that these geometries are $S^3$, $S^2\times\bR$, $\bR^3$, nil, sol, $\widetilde{PSL_2(\bR)}$, $\bH^2\times\bR$, and $\bH^3$.  The purpose of this paper is to develop some theory of pro-$p$ and pro-nilpotent groups which will allow us to understand the group theoretic properties of the fundamental groups of compact manifolds modeled on these geometries.

Throughout this paper, fix a prime $p$.  Let $M=M^3$ be a fibered $3$-manifold.  By this we shall consistently mean that there is a fibration \[\Sigma\to M\to S^1,\] where $\Sigma$ is a topological surface.  Once we are given a fibered $3$-manifold, we will fix a fibration on it once and for all.  We will assume that $\Sigma$ is obtained from a closed surface of some genus $g\geq 1$ by puncturing at finitely many points.  We will denote the monodromy of the fibration by $\psi$.  When it is necessary to prevent confusion, we will write $M_{\psi}$ for the fibered manifold with monodromy $\psi$.  We get a corresponding short exact sequence of groups \[1\to\pi_1(\Sigma)\to\pi_1(M)\to\bZ\to 1.\]  In notation which will persist throughout this paper, the copy of $\bZ$ is generated by a symbol $t$ which we call the {\bf stable letter}.  Since $\pi_1(\Sigma)$ is residually nilpotent, it follows that $\pi_1(M)$ is residually finite solvable.  Sometimes $\pi_1(M)$ might be residually nilpotent or even residually $p$, for instance when the monodromy of the fibration is the identity.  More generally, if the monodromy is a finite order mapping class, $\pi_1(M)$ will be virtually residually $p$  One of our goals will be to understand the conditions under which $\pi_1(M)$ will virtually residually $p$.

We will try to understand the failure or success of $\pi_1(M_{\psi})$ to be residually $p$ from the data of the mapping class $\psi$ and how it acts on the homology of the fiber.  We prove:

\begin{thm}\label{t:rp1}
Let $G=\pi_1(M)$, where $M$ is fibered with fiber $\Sigma$ and monodromy $\psi$.  If $\psi$ acts unipotently on $H_1(\Sigma,\bZ/p\bZ)$, then $G$ is residually $p$.  In particular, if $\psi$ is arbitrary then $G$ is virtually residually $p$.  If $\Sigma$ is a torus, then $G$ is residually $p$ if and only if $p$ divides the determinant of $A-I$, where $A\in SL_2(\bZ)$ is the matrix which expresses the action of $\psi$ on the homology of the torus.
\end{thm}
We shall show that it is generally not necessary for $\psi$ to act unipotently modulo $p$ on the homology for $G$ to be residually $p$.

We will then use recent work of Wise prove the following theorem, which positively answers a question due to Lubotzky:

\begin{thm}\label{t:hypgeneral}
Let $G< PSL_2(\bC)$ be a lattice.  Then $G$ is virtually residually $p$.
\end{thm}
In particular, the conclusion holds for the fundamental group of any finite volume hyperbolic $3$-manifold.  Whereas general theory of finitely generated linear groups gives us the conclusion for all but finitely many primes, it seems that this fact is particular to hyperbolic orbifolds.

In \cite{Wi}, Wilton classifies all compact $3$-manifolds whose fundamental groups are limit groups.  Any such manifold has a residually $p$ fundamental group for each $p$.  However, most of the manifolds considered in this paper are not on the list of residually free $3$-manifolds.  He also asks which compact $3$-manifolds have residually torsion-free nilpotent fundamental groups, a question which we shall explore later in the paper.  Note that any such group is residually $p$ for all $p$.  As for Wilton's question:

\begin{thm}\label{t:rtfn}
Suppose $M$ is a fibered manifold with fiber $\Sigma$ and suppose that the monodromy of the fibration acts unipotently on the homology of $\Sigma$.  Then $\pi_1(M)$ is residually torsion-free nilpotent.
\end{thm}

In order to gain a deeper appreciation for the residual properties of $3$-manifold groups which center around residual nilpotence, we will develop some theory of residually $p$ groups in order to understand the connection between residually $p$ and residually torsion-free nilpotent groups.

This paper stems from an attempt to understand the action of a mapping class on the homology of a surface, or more generally on the homology of a finite cover.  Of particular interest was how mapping classes in the Torelli group $I(\Sigma)$ act on finite covers.  The hope was initially that the residual properties of the resulting mapping torus would give some insight, though it did not.  We do however obtain the following:

\begin{thm}\label{t:rp3}
Suppose $G_{\psi}=\pi_1(M_{\psi})$ is residually $p$ and let $\phi$ be in the kernel of the modulo $p$ homology representation $\Mod(\Sigma)\to Sp_{2g}(\bZ/p\bZ)$.  Then $G_{\psi\circ\phi}$ is residually $p$.
\end{thm}

\section{Acknowledgements}
This work has benefitted from conversations with Ian Agol, Nir Avni, Jason Behrstock, Tom Church, Benson Farb, Stefan Friedl, Alex Lubotzky, Curt McMullen and Ben McReynolds.  The idea for this paper arose from a conversation with Will Cavendish and Renee Laverdiere.  The author has been informed that in \cite{AF1} and \cite{AF2}, Matthias Aschenbrenner and Stefan Friedl have independently obtained many of the results in this paper.  The author was supported by an NSF Graduate Research Fellowship for part of this research.  The author is indebted to the referee for numerous helpful comments and suggestions.

\section{Tools for analyzing $p$-groups}\label{s:tools}
We will record with proofs some tools from the theory of $p$-groups which will be at the heart of many of the arguments in this paper.  The following is a standard fact about $p$-groups:

\begin{lemma}\label{l:tower}
Let $G$ be a group and $K_1,K_2$ two $p$-power index normal subgroups.  Then $G/(K_1\cap K_2)$ is a $p$-group.
\end{lemma}
\begin{proof}
Since $K_1$ and $K_2$ normalize each other, we can form the group $H=K_1K_2$.  Evidently $H$ is normal in $G$ and has $p$-power index.  By the second isomorphism theorem for groups, we have that $H/K_1\cong K_2/(K_1\cap K_2)$.  The left hand side is obviously a $p$-group, and we also have $G>K_2>(K_1\cap K_2)$, whence $G/(K_1\cap K_2)$ is a $p$-group.
\end{proof}

Without the assumption that at least $K_1$ is normal, the intersection $K_1\cap K_2$ may fail to have $p$-power index as can be seen by considering $A_4<A_5$ and noting that the intersection of all conjugates of $A_4$ is trivial.

If $G$ is a finitely generated group, then it has a maximal subgroup.  Recall that the intersection of all maximal subgroups of $G$ is called the {\bf Frattini subgroup} of $G$, and is denoted $\varphi(P)$.  The Frattini subgroup of any group is obviously characteristic.

\begin{lemma}
Let $P$ be a $p$-group.  Then $P/\varphi(P)$ is the largest elementary abelian quotient of $P$.
\end{lemma}
\begin{proof}
Every subgroup of index $p$ in $P$ is normal, so that $P/\varphi(P)$ is at least as large as the largest elementary abelian quotient of $P$.  Suppose $M<P$ is a maximal subgroup of index $p^k$, $k>1$.  We may suppose that $[P,P]$ is not contained in $M$, for otherwise we would have the claim.  We must have that the image of $M$ in $P^{ab}$ is everything, for otherwise $M$ would not be maximal.  On the other hand, in a nilpotent group every maximal subgroup is normal of index $p$ for some prime (cf. \cite{KuSt}).
\end{proof}

The following lemma is important for proving a partial converse to Theorem \ref{t:rp1}:
\begin{lemma}
Let $P$ be a $p$-group whose abelianization is cyclic.  Then $P$ is cyclic.
\end{lemma}
\begin{proof}
If $P^{ab}$ is cyclic then $P/\varphi(P)\cong\bZ/p\bZ$.  Choose a generator $x$ for $P/\varphi(P)$.  Then $P$ is generated by $x$ and $\varphi(P)$.  A standard fact about the Frattini subgroup is that it consists of nongenerators of $P$.  Indeed, let $y\in\varphi(P)$ and let $S=\{x\}\cup\varphi(P)$.  If $S\setminus\{y\}$ does not generate $P$ then it generates a proper subgroup, which is in turn contained in a maximal proper subgroup $P'$.  But $y\in P'$ by definition.
\end{proof}

For any group $G$, we write $\gamma_i(G)$ for the $i^{th}$ term of the lower central series of $G$, and we write $\cL_i(G)$ for the quotient $\gamma_i(G)/\gamma_{i+1}(G)$.  For a prime $p$, we write $\cL_i(G,p)$ for $\cL_i(G)\otimes\bZ/p\bZ$.

Let $\al\in\Aut(G)$ of order $k$.  We will write $G_{\al}$ for the semidirect product \[1\to G\to G_{\al}\to\bZ/k\bZ\to 1,\] where the stable letter $t$ which generates $\bZ/k\bZ$ acts by $\al$ via conjugation.

The following result is a generalization of the fact that if $\al$ acts trivially on $H_1(G,\bZ)$ then it acts trivially on each quotient $\cL_i(G)$.
\begin{lemma}
Let $P$ be a $p$-group.  Suppose that $\al$ is unipotent as a matrix acting on $\cL_i(P,p)$ for all $i$.  Then $P_{\al}$ is a $p$-group.
\end{lemma}
\begin{proof}
Write $A$ for the associated matrix for the action of $\al$ on $P/\varphi(P)$.  Forming a commutator $[t,v]$ for $v\in P/\varphi(P)$ is the same as applying the matrix $A-I$, which is nilpotent since $A$ acts unipotently on $P/\varphi(P)$.  Similarly, forming the commutator $[t^j,v]$ is the same as applying $A^j-I$, which is also nilpotent.  Furthermore, the operators $A^i-I$ all commute with each other.  So, any nested commutator can be written as
\[
\left(\prod_{i=1}^k (A^i-I)^{n_i}\right) v
\]
for some $v\in P/\varphi(P)$.  By the pigeonhole principle any sufficiently long commutator will be trivial (since $A$ has finite order), so that the image of all sufficiently long commutators is in the kernel of the map \[\cL_i(P)\to\cL_i(P,p).\]

Notice that each $\cL_j(P)$ is a quotients of some free abelian group $\bZ^n$, so we may lift the action of $\al$ (which is given by a matrix $A(j)$) to an integral matrix.  We have that each sufficiently long product of the form
\[
\left(\prod_{i=1}^k (A(j)^i-I)^{n_i}\right)
\]
is an operator which sends $\bZ^n$ to $p\bZ^n$.  It follows that for any $k$, any sufficiently long product of the same form sends $\bZ^n$ to $p^k\bZ^n$, so that any sufficiently long commutator in $t$ and $\gamma_j(P)/\gamma_{j+1}(P)$ is trivial.

An easy induction shows that any sufficiently long commutator in $P_{\al}$ thus lands in $\gamma_i(P)$, showing that $P_{\al}$ is nilpotent.  By assumption, we have a splitting $\bZ/k\bZ\to P_{\al}$, and the image is not normal: in fact for each $1\leq i<k$, there is a nontrivial commutator of the form $[t^i,v]$ for some $v\in P$.  It follows that if $k$ is not a power of the prime $p$, then $P_{\al}$ cannot be a product of its Sylow subgroups and hence not nilpotent.  It follows that $P_{\al}$ is a $p$-group.
\end{proof}

\begin{lemma}\label{l:porder}
Let $G$ be a finite nilpotent group, and suppose that $\psi\in\Aut(G)$ acts unipotently on $H_1(G,\bZ/p\bZ)$.  Then $\psi$ has $p$-power order as an automorphism of $G$.
\end{lemma}
\begin{proof}
We form a semidirect product $G_{\psi}$, using $t$ to denote the stable letter.  Suppose that $A$ is induced by $\psi$ and acts unipotently on $H_1(G,\bZ/p\bZ)$.  We let $\cL(1,p)=H_1(G,\bZ/p\bZ)$, and we will write $\cL(i,p)$ to be \[\gamma_{i-1}(G)/\gamma_i(G)\otimes\bZ/p\bZ.\]  We will consider $\cL(i,p)$ as a quotient of \[\Hom(H_1(G,\bZ/p\bZ),\cL(i-1,p)),\] which we write as \[H^1(G,\bZ/p\bZ)\otimes\cL(i-1,p).\]  The quotient map is given by the Lie bracket (cf. \cite{MKS}, \cite{M}).  Consider the reduction of $\cL_i(G)=\gamma_{i-1}(G)/\gamma_i(G)$ modulo $p$.  Elements of $\cL_i(G)$ are finite sums of simple tensors, which are images of simple tensors in $H^1(G,\bZ)\otimes\cL_{i-1}(G)$.  The simple tensors which persist after reducing modulo $p$ can be written so that no $p$-multiple of a cohomology class of $G$ occurs in the tensor.  It follows that the canonical map $\cL_i(G)\to(\cL_i(G)\pmod p)$ factors through $\cL(i,p)$.

If $\psi$ acts unipotently on $H_1(G,\bZ/p\bZ)$ then it also acts unipotently on the cohomology since the two actions are dual to each other.  We have that $A$ is the monodromy matrix acting on $H_1(G,\bZ/p\bZ)$, $A^*$ its transpose, and $A(i)$ the associated matrix acting on $\cL(i,p)$.  Clearly $1$ is the unique point in the spectrum of both $A$ and $A^*$.  Suppose inductively that $A(i)$ is unipotent.  Then the points of the spectrum of $A(i+1)$ are pairwise products of the points of the spectrum of $A(i)$ and $A^*$, so that $A(i+1)$ acts unipotently on $\cL(i+1,p)$.

Filter $G$ by its lower central series, writing $\cL(i)$ for the $i^{th}$ quotient as above.  Since a sufficiently large power of $A(i)-I$ sends $\cL(i)$ to $p\cL(i)$, we have that forming the commutator $[t,G]$ sufficiently many times will send $G$ to a term arbitrarily deep in its lower central series.  It follows that  $G_{\psi}$ is nilpotent.  But then we must have that $\psi$ has $p$-power order as an automorphism of $G$, since otherwise the nilpotence of $G$ would be contradicted.
\end{proof}

\begin{lemma}
Let $P$ be a $p$-group and $\gam<\Aut(P)$ the group of automorphisms which induce the identity on $P/\varphi(P)$.  Then $\gam$ is a $p$-group.
\end{lemma}
\begin{proof}
This is now immediate since $\gam$ acts trivially and hence unipotently on $\cL_1(P)$.
\end{proof}

We finally need the following well-known result, whose proof we sketch for completeness.
\begin{lemma}
Let $G$ be a surface group.  For each prime $p$, $G$ is residually $p$.
\end{lemma}
\begin{proof}
It is well-known that $G$ is a limit group, namely that it is residually free.  On the other hand, a free group is residually $p$.  Suppose that we have a sequence of bases $\{X_i\}$ for nested free groups \[F_1>F_2>\cdots\] with the assumption that $X_i\cap F_{i+1}=\emptyset$.  Then it can be shown (cf. \cite{LySch}) that the shortest word in $F_i$ has length at least $i$ with respect to $X_1$, so that the filtration exhausts all of $F_1$.  Furthermore, is each $F_i$ is proper and characteristic in $F_{i-1}$, then we automatically have $X_i\cap F_{i+1}=\emptyset$, since there is an automorphism of $F_i$ taking any element of $X_i$ to any other.  To construct the filtration $\{F_i\}$, we successively take homology with $\bZ/p\bZ$ coefficients so that $F_i/F_{i+1}$ are all $p$-groups.
\end{proof}

\section{Fibered manifolds and the proof of Theorem \ref{t:rp1}}
Given the material developed in section \ref{s:tools}, the proof of the first part of Theorem \ref{t:rp1} is not difficult.  Theorem \ref{t:rp1} is already not surprising since the automorphism group of a (topologically) finitely generated pro-$p$ group is already virtually a pro-$p$ group (this can in fact already be deduced from the results in section \ref{s:tools}, cf. \cite{DDMS}).

\begin{proof}[Proof of Theorem \ref{t:rp1}, part 1]
Let $U<GL_n(\bZ/p\bZ)$ be a unipotent subgroup.  Then $U$ is a $p$-group.  Let $P$ be a characteristic $p$-group quotient of $\pi_1(\Sigma)$ and suppose that $\al\in\Aut(P)$ acts unipotently on $H_1(\Sigma,\bZ/p\bZ)$.  Since $U$ is a $p$-group, we have that a $p$-power of $\al$ acts trivially on $P/\varphi(P)$.  But this power of $\al$ has $p$-power order, so that the semidirect product $P_{\al}$ is a $p$-group.  If $\al$ is induced by $\psi$, it follows that $P_{\al}$ is a quotient of $G_{\psi}$.
\end{proof}

From the proofs above, we see that Theorem \ref{t:rp1} is a special case of the following more general result, whose proof is identical:
\begin{thm}
Let $G$ be a finitely generated residually $p$ group and $\gam$ a group of automorphisms of $G$ which act unipotently on $H_1(G,\bZ/p\bZ)$.  Then $G_{\gam}$ is residually $p$, where $G_{\gam}$ is the semidirect product which fits into the following short exact sequence: \[1\to G\to G_{\gam}\to\gam\to 1.\]
\end{thm}

We now turn to torus bundles over the circle, whereby we can give a more complete characterization of when the fundamental group of the bundle is residually $p$.  The following proposition contains the rest of the content of Theorem \ref{t:rp1}.

\begin{prop}
Let $A\in SL_2(\bZ)$ and $M_A$ the associated torus bundle with fundamental group $G_A$.  Then $G_A$ is not residually $p$ for any $p$ if and only if $A$ is conjugate over $\bQ$ to
\[
\begin{pmatrix}
2&1\\1&1
\end{pmatrix}.
\]
$G_A$ is residually $p$ if and only if $p\mid\det(A-I)$.
\end{prop}
\begin{proof}
Suppose $A-I\in GL_2(\bZ)$, which is the case when \[
A=\begin{pmatrix}
2&1\\1&1
\end{pmatrix}.
\]  Then we can solve the equations $(A-I)v=(1,0)$ and $(A-I)w=(0,1)$ in $\bZ^2$.  If $t$ denotes the monodromy generator, it follows that $[t,\bZ^2]=\bZ^2$, so that the sequence of subgroups $\{\gamma_i(G)\}$ stabilizes at $i=1$ with $\gamma_i(G)=\bZ^2$ for all $i\geq 1$.  It follows in this case that $G_A$ is not even residually nilpotent.

Now suppose that $A-I$ is not invertible over $\bZ$ and let $p$ be a prime dividing the determinant of $A-I$.  Note that over $\bR$, $A-I$ is conjugate, up to a sign, to
\[
X=\begin{pmatrix}
n-2&1\\n-2&0
\end{pmatrix}.
\]
This is because every hyperbolic element of $PSL_2(\bR)$ is determined up to conjugacy by its trace.
So there exists a real matrix $Q$ such that $Q(A-I)=XQ$.  Finding the entries of $Q$ is tantamount to solving a system of linear equations with integer entries.  Since $Q$ is only well-defined up to a scalar matrix, we may assume that one of the entries is equal to $1$.  By Cramer's rule the solutions are rational, so we may assume $Q$ has rational entries.

One verifies that
\[
X^2=
\begin{pmatrix}
(n-2)^2+n-2 & n-2\\(n-2)^2 & n-2
\end{pmatrix}.
\]
In particular, each entry of $X^k$ is divisible by $n-2$ whenever $k\geq 2$.  Choose $p$ a prime dividing $n-2$ and let $\al$ be an entry of $Q$.  Write $\al=a/b$, where $a$ and $b$ are relatively prime integers.  Some power of $p$ divides $b$, and let $m$ be the maximal power of $p$ dividing the denominator of any entry of $Q$ or $Q^{-1}$.  Take $k$ sufficiently large so that every entry of $X^k$ is divisible by $p^{2m}$.  Then $QX^kQ^{-1}=(A-I)^k$ and we see that every entry of $(A-I)^k$ is divisible by $p$.  It follows that $A-I$ is nilpotent modulo every power of $p$.  It follows that the quotient of $G$ given by reducing the torus homology modulo $p^k$ is nilpotent, so that $G$ is residually $p$.

Conversely, suppose that $p$ does not divide $\det(A-I)$.  It follows that $A-I$ is invertible modulo $p$.  Suppose we want $G_A$ to inject into its pro-$p$ completion, and let $P$ be a $p$-group quotient.  Write $P'$ for the image of $\bZ^2$ in $P$, and write $t$ again for the image of the stable letter in $P$.  $P'$ is abelian and is invariant under the conjugation action of $t$.  Tensoring $P'$ with $\bZ/p\bZ$, we may assume that $P'$ is elementary abelian.  In particular, we have a map $(\bZ/p\bZ)^2\to P'$ which is either a bijection or has rank one.  If it is a bijection, then $A-I$ acts on $P'$ as an element of $GL_2(\bZ/p\bZ)$ since $A-I$ is invertible modulo $p$, and in particular it does not act nilpotently.  It follows that then $P$ could not have been nilpotent.  If there is a nontrivial kernel in the map  $(\bZ/p\bZ)^2\to P'$, then $P$ is nilpotent if and only if some power of $A-I$ has rank $\leq 1$ in $\End((\bZ/p\bZ)^2)$.  This is again a contradiction since $p$ does not divide $\det(A-I)$.
\end{proof}

Applying Theorem \ref{t:rp1}, we obtain:
\begin{cor}
Let $A\in SL_2(\bZ)$ and $p$ a prime.  Then there are infinitely many $k$ such that $p\mid\det(A^k-I)$.
\end{cor}

We will see in section \ref{s:mod} that it seems unlikely that we can obtain a complete description of mapping classes which give rise to residually $p$ fibered manifolds which are not torus bundles.

\section{Higher torus bundles}
In this section we prove a simple analogue of Theorem \ref{t:rp1} for higher torus bundles.  The setup is the same as before, only the monodromy matrix $A$ sits in $GL_n(\bZ)$.  We denote the torus bundle by $M_A$ and its fundamental group by $G_A$.

From the proof of the second part of Theorem \ref{t:rp1}, we immediately obtain see that a torus bundle whose fundamental group $G_A$ is not residually $p$ must have $\det(A-I)$ not divisible by $p$.

On the other hand:
\begin{prop}\label{p:torus}
$G_A$ is residually $p$ if and only if $A$ is unipotent modulo $p$.
\end{prop}
In the proof we will use the fact that $p$-groups are closed under taking products.  Therefore if we enumerate any finite set $A$ of elements in a residually $p$ group $G$, we may find a sequence of nested normal subgroups of $G$ of $p$-power index which together witness the nontriviality of each element of $A$.
\begin{proof}[Proof of Proposition \ref{p:torus}]
The ``if" direction is trivial in light of the theory we have developed already.  So, suppose that $A$ is not unipotent modulo $p$, let $P$ be a $p$-group quotient of $G_A$ and let $P'$ be the image of $\bZ^n$ in $P$.  After tensoring with $\bZ/p\bZ$, we again assume that $P'$ is an elementary abelian $p$-group.  We have a map $(\bZ/p\bZ)^n\to P'$ which is $A$-equivariant and has some kernel $K$.  Since $P$ is nilpotent, we must have that the $(A-I)$-action on $P'$ is nilpotent.  In particular, some power of $A-I$ must send $(\bZ/p\bZ)^n$ to $K$.

If $v\in\bZ^n$ maps nontrivially to $K$, then we must have a $p$-group quotient $Q$ of $G_A$ in which $v$ is nontrivial.  By applying a sufficiently large power of $A-I$, we may assume that the image $Q'$ of $\bZ^n$ in $Q$ maps trivially to $P'$.  Since $Q$ is nilpotent, again we must have that the action of $A-I$ is nilpotent.  Since $v$ was chosen arbitrarily, we may proceed in this fashion until all of $K$ is exhausted.  Precisely, this means that a sufficiently large power of $A-I$ will have to send $\bZ^n$ to $p\bZ^n$.  It follows that $A$ acts unipotently on $(\bZ/p\bZ)^n$.
\end{proof}

\section{Generalizations to Baumslag-Solitar groups}
Recall that the $(p,q)$-Baumslag-Solitar group $\gam_{p,q}$ is defined be the presentation \[\gam_{p,q}=\langle s,t\mid st^ps^{-1}=t^q\rangle.\]  For more details, see \cite{dlH} and the references therein.  It is well-known that if either $p=1$, $q=1$, or $p=q$ then $\gam_{p,q}$ is residually finite and therefore Hopfian.  In general, $\gam_{p,q}$ is not even Hopfian.  It was in fact Baumslag who observed that the group \[\gam_{2,3}\cong\langle s,t\mid (t^2)^st^{-3}\rangle\] was non-Hopfian (since $\phi:t\mapsto t^2$, $\phi:s\mapsto s$ is obviously a surjection, but $\phi([t,t^s])=[t^2,t^3]=1$ and $[t,t^s]$ is nontrivial in $\gam_{2,3}$).  It is Hopfian if and only if $p$ and $q$ share precisely the same set of prime divisors or if one divides the other.  To avoid trivialities we will assume $(p,q)\neq (1,1)$ and that $p,q\geq 1$.

When $A\in\Aut(\bZ^n)$, we defined the semidirect product of $\bZ^n$ and $\bZ$ using $A$ and studied it as the fundamental group of a torus bundle over the circle.  We can easily replace $\Aut(\bZ^n)$ with $\End(\bZ^n)$ and construct similar semidirect products.  When $n=1$, we get the class of $(1,q)$-Baumslag-Solitar groups.  These groups can be viewed as the fundamental groups of mapping tori of endomorphisms of the circle, namely the endomorphisms given by $z\mapsto z^q$.  We obtain the following proposition, whose proof is immediate in view of all the work we have done:

\begin{prop}
Let $A\in\End(\bZ^n)$ and let $G_A$ denote the corresponding semidirect product.  Then $G_A$ is $\omega$-nilpotent if and only if \[\bigcap_i (A-I)^i(\bZ^n)=\{0\}.\]
\end{prop}
\begin{cor}
The Baumslag-Solitar group $\gam_{1,q}$ is $\omega$-nilpotent if and only if $q\neq 2$.  In particular, $\gam_{1,q}$ is residually finite solvable for all $q$ and residually nilpotent if $q\neq 2$.  It is residually $p$ at exactly the primes dividing $q-1$.
\end{cor}

We will see later that no Baumslag-Solitar group is torsion-free nilpotent as a corollary to Proposition \ref{p:tfn}.

\section{Finite-volume hyperbolic orbifolds}
The main external tool for proving Theorem \ref{t:hypgeneral} is the following theorem due to Wise in \cite{W1} (cf. \cite{W2}, \cite{Wk}):

\begin{thm}\label{t:wisemain}
Let $G$ be a word-hyperbolic group with a quasiconvex hierarchy.  Then $G$ has a finite index subgroup $G'$ which embeds in a graph group $R$.
\end{thm}

The precise meanings of all the terms in Theorem \ref{t:wisemain} are not important for our purposes.  However, two important corollaries of this theorem are the following:
\begin{cor}
Let $M$ be a finite volume cusped hyperbolic manifold, and assume $M$ contains a geometrically finite incompressible surface $S$.  Then $G=\pi_1(M)$ has a finite index subgroup $G'$ which embeds in a right-angled Artin group.
\end{cor}

A Haken hierarchy for $M$ implies the existence of a hierarchy for $\pi_1(M)$.  By the work of Thurston, this is a quasiconvex hierarchy if and only if $S$ is geometrically finite (see \cite{T}, cf. \cite{S}).

\begin{cor}
Every Haken hyperbolic $3$-manifold is virtually fibered.
\end{cor}
\begin{proof}
By the work of Bonahon in \cite{B}, an incompressible surface $S$ is either a virtual fiber, or $S$ is geometrically finite.  In the first case, $M$ virtually fibers.  In the second case, there is a finite cover $\widehat{M}$ of $M$ such that $\pi_1(\widehat{M})<R$ for a graph group $R$.  Each graph group is residually torsion-free nilpotent, hence residually finite rationally solvable.  By \cite{A2}, $M$ virtually fibers.
\end{proof}

Let $\gam$ be as in the statement of Theorem \ref{t:hypgeneral}.  Any such $\gam$ is the fundamental group of a hyperbolic orbifold.  After passing to a finite cover, we may assume that $\gam$ is the fundamental group of a hyperbolic $3$-manifold $M$ of finite volume.  By a result due to Thurston, the representation $\pi_1(M)\to PSL_2(\bC)$ lifts to $SL_2(\bC)$ (cf. \cite{CS}).

Let $\mathcal{R}=\mathcal{R}(\gam)$ denote the $SL_2(\bC)$ representation variety of $\gam$.  By general theory (see \cite{R} for instance, though this theory goes back to the work of Mal'cev), $\mathcal{R}$ contains a point over $\overline{\bQ}$ and in fact a faithful representation $\gam\to SL_2(\overline{\bQ})$.  Since $\gam$ is finitely generated, there is a finite extension $K/\bQ$ such that the image of $\gam$ lands in $SL_2(K)$.  We let $\mathcal{O}$ denote the ring of integers in $K$.  In any matrix in the image, there are at most four denominators, and so any finite generating set for $\gam$ has only finitely many denominators occurring among nonzero entries in its image.  Fix a finite generating set for $\gam$ and consider the denominators which occur.  These will be contained in finitely many prime ideals in $\mathcal{O}$.  Each prime ideal of $\mathcal{O}$ lies over a unique prime ideal $p\bZ$.  For the set of denominators which occur in the image a generating set for $\gam$, let $\mathcal{B}\subset\bZ$ be the finite set of primes over which the associated prime ideals in $\mathcal{O}$ lie.  We call $\mathcal{B}$ the set of {\bf bad primes}.

\begin{lemma}
Let $p$ be any prime.  When $\mathcal{B}$ is empty, $\gam$ is virtually residually $p$.
\end{lemma}
\begin{proof}
This is entirely analogous to the fact that $SL_2(\bZ)$ is virtually residually $p$ and is done using the first congruence subgroup.  Let $P$ lie over $p\bZ$.  Let $\gam_1$ denote the kernel of the natural map \[SL_2(\mathcal{O})\to SL_2(\mathcal{O}/P).\]  We have a natural action of $\gam_1$ on $\mathcal{O}^2$, and we can construct the semidirect product \[1\to\mathcal{O}^2\to G\to\gam_1\to 1.\]  We can also construct truncated semidirect products of the form \[1\to(\mathcal{O}/P^n)^2\to G_n\to\gam_1\to 1.\]  By considering the successive quotients $(P^i/P^{i+1})^2$, we see that the conjugation action of $\gam_1$ on $(\mathcal{O}/P^n)^2$ is unipotent.  Let $K_n<\gam_1$ denote the kernel of this action.  We have that $\mathcal{O}/P^n$ is always a $p$-group.  It follows that the semidirect product \[1\to(\mathcal{O}/P^n)^2\to\overline{G}_n\to\gam_1/K_n\to 1\] is a $p$-group, so that \[\bigcap_n K_n=\{1\}\] and $\gam_1/K_n$ is a $p$-group for all $n$ (a similar argument is fleshed out in section \ref{s:tools}, cf. \cite{BL}).
\end{proof}

Alternatively, the argument could have proceeded as follows: $(\mathcal{O}/P^n)^2$ is a $p$-group, and $(P/P^n)^2$ is its Frattini subgroup.  On the other hand, we have seen that if $Q$ is a $p$-group and $\phi(Q)$ is its Frattini subgroup, then the group of automorphisms of $Q$ which induce the identity on $Q/\phi(Q)$ form a $p$-group, whence the conclusion.

We now consider the case where $\mathcal{B}\neq\emptyset$.  Fix $P\in\mathcal{B}$, and let $\whO=\whO_P$ be the completion of $\mathcal{O}$ at $P$, namely \[\whO=\varprojlim \mathcal{O}/P^n.\]  We let $\widehat{K}$ be the fraction field of $\whO$.  We have a canonical map $\mathcal{O}\to\whO$ which is injective since $\mathcal{O}$ is a Dedekind domain, and thus we have an injective map $K\to\widehat{K}$, and a faithful representation $\gam\to SL_2(\widehat{K})$ induced by the inclusion $SL_2(K)\to SL_2(\widehat{K})$.  By construction, the field $\widehat{K}$ comes equipped with a discrete valuation $\nu$.  Explicitly, it takes an equivalence class of fractions $\gamma=\al/\beta$, determines an $i$ and $j$ such that $\al\in P_i\setminus P^{i+1}$ and $\beta\in P^j\setminus P^{j+1}$ and sets $\nu(\gamma)=i-j$.  By abuse of notation, we write $P$ as the maximal ideal generated by the image of $P$ in $\whO$.  The valuation $\nu$ thus defined is a discrete valuation, so that $\whO$ is a DVR.

Let $V$ be a two-dimensional vector space over $\widehat{K}$.  Recall that at $\whO$-lattice in $V$ is a rank $2$ $\whO$-module which spans $V$ as a $\widehat{K}$-vector space.  Let $L$ be a $\whO$-lattice and $L'$ a sublattice.  Then $L/L'$ is isomorphic to $\whO/P^a\oplus\whO/P^b$ for some choice of nonnegative integers $a$ and $B$.  There is a natural action of $\widehat{K}$ on the set of $\whO$-lattices in $V$.  Note that if $L$ and $L'$ are arbitrary lattices, we can replace $L'$ by an equivalent lattice $kL'$ such that $kL'\subset L$ by choosing an appropriate $k\in\widehat{K}$.  We declare the $\widehat{K}$-orbits to be equivalence classes.  There is a natural graph whose vertices are equivalence classes of lattices, and whose edges span pairs of equivalence classes for which there exist representatives satisfying $L/L'\cong \whO/P$.  It is shown in \cite{Se} that this graph is a tree, called the lattice tree of $\whO$.

We have that $SL_2(\widehat{K})$ acts on this tree in the obvious way.  The stabilizers of vertices are precisely the $GL_2(\widehat{K})$-conjugates of $SL_2(\whO)$ in $SL_2(\widehat{K})$.  The following lemma follows easily from this discussion.

\begin{lemma}
If $\gam$ is as above, then either $\gam$ is virtually residually $p$ or $\gam$ acts on the lattice tree of $\whO$ without a global fixed point.
\end{lemma}

General Bass-Serre theory (cf. \cite{Se}) therefore implies that when $\gam$ acts nontrivially on the lattice tree, then $\gam$ splits as a nontrivial amalgamated product.  Furthermore, the amalgamating group can, up to conjugacy be taken to be the image of $\gam$ in $SL_2(\whO)$.  Since $\gam$ is the fundamental group of a hyperbolic manifold, the amalgamating group is nontrivial.  Indeed, $\bH^3/\gam$ is irreducible and hence $\gam$ cannot split as a nontrivial free product.

The final ingredient we need is the following, which is due to Epstein, Stallings and Waldhausen, and a proof can be found in \cite{CS}.
\begin{lemma}
Let $M$ be a compact, orientable $3$-manifold.  For any nontrivial splitting of $\pi_1(M)$ there exists a nonempty system $S$ of incompressible non-peripheral surfaces such that the image of the inclusion on fundamental groups is contained in an edge group.  Furthermore, the image of the fundamental groups of the components of $M\setminus S$ are contained in a vertex group.
\end{lemma}

This lemma applies to our situation, since by \cite{A1} and \cite{CG}, $\bH^3/\gam$ is homeomorphic to the interior of a compact $3$-manifold.  From here, Theorem \ref{t:hypgeneral} is obvious:

\begin{proof}[Proof of Theorem \ref{t:hypgeneral}]
If $\gam$ cannot be coaxed into admitting a faithful representation into $SL_2(\mathcal{O})$, we have that $\mathcal{B}\neq\emptyset$.  For each $p\in\mathcal{B}$, choose $P\subset\mathcal{O}$ lying over $p$.  We constructed a faithful representation of $\gam$ into $SL_2(\widehat{K})$ whose image does not lie in $SL_2(\whO_P)$.  But then we obtain a nontrivial splitting of $\gam$, and conclude that $\bH^3/\gam$ is Haken.  By a corollary to Theorem \ref{t:wisemain}, it follows that $\bH^3/\gam$ virtually fibers.  We have already shown that any $3$-manifold which virtually fibers over the circle has a virtually residually $p$ fundamental group.
\end{proof}

We briefly remark that whereas general theory says that a finitely generated subgroup of a linear group over a characteristic zero field is always residually $p$ for all but finitely many primes, it may not be residually $p$ for all primes.  An example, furnished in \cite{W}, is the metabelian group generated by \[A=\begin{pmatrix}2&0\\ 0&1\end{pmatrix},\, B=\begin{pmatrix}1&0\\1&1\end{pmatrix}\] contains a subgroup isomorphic to the dyadic rationals and hence cannot be residually $p$ at $2$.

\section{Applications to the mapping class group and the proof of Theorem \ref{t:rp3}}\label{s:mod}
The original motivation for the work in this paper was to study homological representation theory of the mapping class group, as initiated in \cite{K}.  In that paper, the author constructs an infinite dimensional representation $H(\Sigma)$ of the marked mapping class group $\Mod^1(\Sigma)$, which is constructed by identifying $\Mod^1(\Sigma)$ with a subgroup of $\Aut(\pi_1(\Sigma))$, taking any family of characteristic covers of the surface $\Sigma$, and letting the mapping classes act on the homology of these covers.  The content of \cite{K} is that if homology is taken with $\bQ$ coefficients and the family of covers exhausts $\pi_1(\Sigma)$, then the representation is faithful and detects the Nielsen-Thurston classification of each mapping class.

The representation $H(\Sigma)$ is rather difficult to understand, and the hope was that the fundamental group of $M_{\psi}$ would lead to some insight into the action of $\psi$ on $H(\Sigma)$.  One reason for this belief is Theorem \ref{t:rp3}, which says that if $G_{\psi}$ is residually $p$ and $\phi$ is in the Torelli group, then $G_{\psi\circ\phi}$ is also residually $p$.  In other words, whether or not $G_{\psi}$ is residually $p$ depends only on the image of $\psi$ in the symplectic representation $\rho:\Mod(\Sigma)\to Sp_{2g}(\bZ)$.  One would expect that if a good criterion for $G_{\psi}$ to be residually $p$ could be devlopped and could be read off from $\rho(\psi)$, one might be able to restrict the possible actions of $\psi$ on $H(\Sigma)$.  We will now provide a proof of Theorem \ref{t:rp3} and then give some examples which illustrate that a good criterion might not exist.

\begin{proof}[Proof of Theorem \ref{t:rp3}]
This is a manifestation of the fact that for any $p$-group $P$, the group of automorphisms which act trivially on $P/\varphi(P)$ is a $p$-group.  For every $p$-group quotient $P$ of $G_{\psi}$, we shall produce a strictly larger $p$-group quotient $Q$ of $G_{\psi\circ\phi}$, in the sense that the kernel of the map $\pi_1(\Sigma)\to P$ contains the kernel of the map $\pi_1(\Sigma)\to Q$.

Let $P$ be a finite $p$-group quotient of $G_{\psi}$ and let $P'$ be the image of $\pi_1(\Sigma)$ in $P$.  Then $P'/\varphi(P')$ is a quotient of $H_1(\Sigma,\bZ/p\bZ)$.  We have that $\psi$ acts on $H_1(\Sigma,\bZ/p\bZ)$, and the action of $\psi\circ\phi$ on $H_1(\Sigma,\bZ/p\bZ)$ coincides with the action of $\psi$.  Since the stable letter of $P$ has $p$-power order, a $p$-power of $\psi$ acts trivially on $P'/\varphi(P')$.

Let $K$ be the kernel of the quotient map $\pi_1(\Sigma)\to P'$.  Clearly $K$ may not be stabilized by $\psi\circ\phi$, but $K$ has only finitely many conjugates under $\psi\circ\phi$.  Their intersection $K'$ is an intersection of normal $p$-power index subgroups of $\pi_1(\Sigma)$, and hence $Q'=\pi_1(\Sigma)/K'$ is a $p$-group.  On the other hand, $K'$ is contained in a $\psi$- and $\phi$-invariant subgroup of $\pi_1(\Sigma)$, namely the kernel of the map $\pi_1(\Sigma)\to P'/\varphi(P')$.  Consider $Q'/\varphi(Q')$, which is a quotient of $H_1(\Sigma,\bZ/p\bZ)$ as well.  We claim that $Q'/\varphi(Q')$ is canonically isomorphic to $P'/\varphi(P')$.  This is obvious since $\phi$ acts trivially on $H_1(\Sigma,\bZ/p\bZ)$, so that $\psi\circ\phi$ acts trivially on the kernel of the map $H_1(\Sigma,\bZ/p\bZ)\to P'/\varphi(P')$.  

It follows that a $p$-power of $\psi\circ\phi$ acts trivially on $Q'/\varphi(Q')$.  Let $Q$ be the semidirect product constructed from the action of $\psi\circ\phi$ on $Q'$.  It follows that $G_{\psi\circ\phi}$ is residually $p$ since the kernel of the map $\pi_1(\Sigma)\to P$ contains the kernel of the map $\pi_1(\Sigma)\to Q$.
\end{proof}

Note that if $\phi$ is as in the hypotheses of the theorem, it is immediate that $G_{\psi}$ is residually $p$ if and only if $G_{\psi\circ\phi}$ is residually $p$.

Next we analyze an example to show that we cannot expect a good converse to Theorem \ref{t:rp1} when $\Sigma$ is not a torus.  Recall that the braid group on $n$-strands $B_n$ is identified with the mapping class group $\Mod(\Sigma_{0,n},\partial)$, the compactly supported mapping class group of the $n$-times punctured open disk with compactly supported isotopies.  The pure braid group $P_n$ is the kernel of the homology representation $B_n\to S_n$.

By Theorem \ref{t:rp1}, if $\beta\in P_n$ then $G_{\beta}$ is residually $p$ for all primes.  Consider the standard braids $\sigma_1,\sigma_2\in B_3$.  Let $\beta=\sigma_1\sigma_2^{-1}$.  Then $\beta$ is Thurston's example of the simplest pseudo-Anosov braid.  The thrice-punctured disk admits a double cover $\Sigma$ which is a four-times punctured torus, and such that $\beta$ acts on the homology by a block matrix which consists of a \[\begin{pmatrix} 2&1\\ 1&1\end{pmatrix}\] block and a permutation matrix block.

We see that $\beta^3$ is a pure braid, so that the fundamental group of its suspension $G_{\beta^3}$ is residually $p$ for all primes.  Being residually $p$ is inherited by subgroups, so suspending $\beta^3$ and its action on $\Sigma$ gives us a subgroup $G'<G_{\beta^3}$ which is residually $p$ for all primes.  An easy computation shows that \[\begin{pmatrix} 2&1\\ 1&1\end{pmatrix}^3=\begin{pmatrix} 13&8\\ 8&5\end{pmatrix}=A.\]  The only prime which divides $\det(A-I)$ is $2$, so $A$ has a chance of being unipotent only modulo $2$ (it is, as predicted by the proof of the second part of Theorem \ref{t:rp1}).  Just from looking at the action of $\beta^3$ on the homology of $\Sigma$, it is not at all obvious why $G'$ should be residually $p$.

This same example also shows that if $G_{\psi}$ is residually $p$ for all primes then it is not necessarily true that $\psi$ acts unipotently on $H_1(\Sigma,\bZ)$.  This can also be seen in \cite{K} where the author shows that a pseudo-Anosov pure braid may act with positive entropy on the homology of a finite cyclic cover of the disk which is equally branched over all the punctures.

Seeing as there is no good converse to Theorem \ref{t:rp1} and in light of the example above, we are led to the question of whether there exist fibered hyperbolic manifolds whose fundamental groups are not already residually $p$ for all primes.  Stefan Friedl informs the author that there are examples of fibered knot complements whose fundamental groups are not even residually nilpotent, and that the methods used to probe this property can be traced back to \cite{S}.  Here we will produce some explicit examples.

\begin{prop}
Let $M=M_{\psi}$ be a fibered manifold whose fundamental group is not residually nilpotent, and suppose $M$ is not a torus bundle.  Then the fiber cannot be a punctured torus.
\end{prop}
\begin{proof}
It is well-known that the mapping class group of the once punctured torus is $B_3$, the braid group on three strands (cf. \cite{BiHi}).  The point is that each element of $B_3$ has prime order under the homology representation $B_3\to S_3$.  If $\psi\in B_3$ is arbitrary, we let $p$ be its order in $S_3$.  Since $\pi_1(M_{\psi^p})$ is residually $p$, we obtain a $p$-order extension of a residually $p$ group, which is obviously residually $p$.
\end{proof}

Above we produced a fibered $3$-manifold with fundamental group $G'$, and noted that $G'$ injects into all of its pro-$p$ completions.  The fact that $G'$ has this property even though it does not obviously have to have this property is at least partially a reflection of the fact that the action of $\beta^3$ on the homology of the multiply punctured torus is highly reducible.  In order to produce suspensions of mapping classes whose fundamental groups are not residually $p$ for all $p$, the suspended mapping class should be minimal in its commensurability class.  Namely, the mapping class should not be a proper power and should not commute with any finite group of homeomorphisms of the fiber.

\begin{prop}
The suspension of the braid $\beta$ is residually $p$ at $3$ and at no other prime.
\end{prop}
\begin{proof}
It is easy to see that $\beta$ acts as an order $3$ permutation on the homology of the fiber $\Sigma$, and that rationally the homology of the fiber splits into a $1$-dimensional and a $2$-dimensional irreducible representation.  The $2$-dimensional irreducible representation of $\beta$ may decompose further modulo a prime, but the further irreducible representations cannot be trivial and hence must be equipped with a faithful $\beta$-action.  Suppose that $G_{\beta}$ injects into its pro-$p$ completion, so that we have an injective map $c_p:G_{\beta}\to\whG_{\beta,p}$.  Consider the restriction of $c_p$ to $\pi_1(\Sigma)$.  Clearly there is a $p$-group quotient $P$ of $G_{\beta}$ such that the image $Q$ of $\pi_1(\Sigma)$ is nonabelian since $c_p$ is injective.  We have that $Q^{ab}$ is a quotient of $H_1(\Sigma,\bZ_p)$ and is also a $\beta$-module.  Tensoring with $\bZ/p\bZ$, we see that $Q^{ab}$ cannot be cyclic, since a $p$-group with cyclic abelianization is cyclic.  It follows that $Q^{ab}$ is a subgroup of a quotient of $P$, equipped with an order $3$ automorphism given by $\beta$.  It follows $P$ cannot be a $p$-group unless $p=3$.
\end{proof}

By considering other transitive braids on disks with more punctures, we can obtain groups which inject into exactly one pro-$p$ completion for primes other than $3$.

\begin{prop}
Let $\al$ be the homeomorphism of the once-punctured torus given by lifting $\beta$ to a double cover of the thrice punctured disk and filling in all but one puncture.  Let $G_{\al}$ be the fundamental group of the suspension of $\al$.  Then $G_{\al}$ is not residually nilpotent.
\end{prop}
\begin{proof}
The proof is almost identical to the previous proposition.  The action of $\al$ on the homology of the once punctured torus is by the matrix \[\begin{pmatrix} 2&1\\ 1&1\end{pmatrix}.\]  There is a finite nilpotent quotient $N$ of $G_{\al}$ such that the image $K$ of $\pi_1(\Sigma)$ is nonabelian.  Since a finite nilpotent group is a direct product of its Sylow $p$-subgroups, $K$ admits a quotient which is a nonabelian $p$-group $P$, and further a homomorphism to a noncyclic elementary abelian $p$-group $(\bZ/p\bZ)^2$.  We have observed that $\al$ does not act on $(\bZ/p\bZ)^2$ unipotently, so that $N$ cannot be nilpotent.
\end{proof}

\section{Central extensions of residually $p$ groups}
The one geometry we have not considered to this point is $\widetilde{PSL_2(\bR)}$ geometry.  If $\gam$ is a finitely generated discrete subgroup of isometries of this geometry then $\gam$ fits into a short exact sequence of groups as follows (cf. \cite{T2}): \[1\to\bZ\to\gam\to H\to 1.\]  Furthermore, the copy of $\bZ$ in the sequence is central.  The group $H$ is a discrete subgroup of isometries of $\bH^2$, so that virtually $\gam$ is a $\bZ$-central extension of the fundamental group of a hyperbolic surface and hence (virtually) the fundamental group of an orientable circle bundle over a surface.  Recall that central extensions of a group $G$ by a group $A$ have exactly one obstruction to triviality, namely the Euler class, which is an element of $H^2(G,A)$ (see \cite{Br}).  It follows immediately that:
\begin{prop}
Let $H$ be as above, and suppose that $H$ is virtually $\pi_1(\Sigma)$.  If $\Sigma$ is not closed then $\gam$ virtually splits as a product of $\bZ$ and a free group.  In particular, $\gam$ is the fundamental group of a virtually fibered $3$-manifold and is hence virtually residually $p$.
\end{prop}

If $\Sigma$ is closed, it is not immediate that $\gam$ is residually $p$.  In fact, we have a more general question:
\begin{quest}
Let $P$ be a finite $p$-group which fits into a short exact sequence \[1\to P\to G\to Q\to 1,\] where $Q$ is residually $p$.  Suppose that the conjugation action of $G$ on $P/\varphi(P)$ is unipotent.  Under what conditions is $G$ residually $p$?
\end{quest}

Since $G$ acts unipotently on $P/\varphi(P)$ and since $P$ is finite, we obtain a $p$-power index normal subgroup $G'$ of $G$ such that \[1\to P\cap G'\to G'\to Q'\to 1\] is a central extension.  In particular $P\cap G'$ is abelian.  We can also consider the case where $P$ is replaced by a torsion-free abelian group and the extension is central.

Notationally, we replace $G'$ by $G$.  Suppose that $\gamma\in P$ and we want a finite $p$-group quotient $P'$ of $G$ such that $\gamma$ survives in the quotient.  If $P$ is finite, we may assume that the central extension \[1\to P\to G\to Q\to 1\] descends to another central extension \[1\to P\to P'\to Q'\to 1.\]  Here, $Q'$ is a $p$-group quotient of $Q$.  We have that the extension $P'$ is classified by an element of $H^2(P',P)$.

We have a natural pullback map $H^2(P',P)\to H^2(Q,P)$, and we can hope for $G$ to be residually $p$ only if the classifying cocycle for \[1\to P\to G\to Q\to 1\] is contained in the image of $H^2(P',P)$.  Let us be a bit more explicit about this fact.  The standard discussion we follow here can be found in \cite{Br}.  Let $G$ be an arbitrary group and let $A$ be an abelian group equipped with the structure of a trivial $G$-module.  It is classical that there is a bijection between the sets $E(G,A)$ of central extensions of $G$ by $A$ and $H^2(G,A)$.  This bijection associates to each cohomology class $c\in H^2(G,A)$ a function $f=f_c:G\times G\to A$ satisfying the cocycle condition, called the {\bf factor function}.  If we are given an extension \[1\to A\to E\to G\to 1,\] we choose a set-theoretic section $s:G\to E$.  We normalize $s$ so that $s(1)=1$.  If $i:A\to E$ is the inclusion map, we define $f$ by $i(f(g,h))=s(g)s(h)s(gh)^{-1}$.  The group law on $A\times G$ can be recovered by $(a,g)(b,h)=(a+b+f(g,h),gh)$.

Let $F$ be a quotient of $E$ with quotient map $\phi$, and suppose that the induced map $A\to F$ is injective.  Since $A$ is central, $\phi$ descends to a quotient map $G\to F/A$ which we also call $\phi$.  The question in which we are really interested here is when a given quotient map $\phi:G\to F/A$ extends to a quotient map $\phi:E\to F$.  Then, we may write \[1\to A\to F\to F/A\to 1,\] and $F/A$ is a quotient of $G$.  A group law on $A\times F/A$ must be given by $(a,\phi(g))(b,\phi(h))=(a+b+k(\phi(g),\phi(h)),\phi(gh))$ for some factor function $k$.  If $F$ is to be a quotient of $E$, we must have $(a+b+k(\phi(g),\phi(h)),\phi(gh))=(a+b+f(g,h),\phi(gh))$.  We have that $f$ determines an element of $H^2(G,A)$ and $k\circ\phi$ determines another element of $H^2(G,A)$ which is a pullback of an element of $H^2(F/A,A)$.  All the equation $(a+b+k(\phi(g),\phi(h)),\phi(gh))=(a+b+f(g,h),\phi(gh))$ says is that $f=\phi^*k$.

The main result of all this discussion is:

\begin{prop}
Let $S^1\to M\to\Sigma$ be a nontrivial orientable circle bundle with trivial monodromy over $\Sigma$ and let $p$ be a prime.  Then $\gam=\pi_1(M)$ is virtually residually $p$.
\end{prop}
In Section \ref{s:rtfn} we shall develop technically simpler characteristic zero machinery to prove this proposition as a corollary to Theorem \ref{t:psl2r}, which will show that $\gam$ is residually torsion-free nilpotent.

\section{Residually torsion-free nilpotent $3$-manifold groups}\label{s:rtfn}
In \cite{Wi}, Wilton classifies the fundamental groups of $3$-manifolds which are residually free, in other words which $3$-manifold groups are limit groups.  He shows that limit groups among $3$-manifold groups are rare:
\begin{prop}[\cite{Wi}]
If $M$ is a prime, compact $3$-manifold with incompressible toral boundary, then $\pi_1(M)$ is free and nontrivial if and only if $M$ is one of the following:
\begin{enumerate}
\item
A trivial circle bundle over an orientable surface.
\item
A circle bundle with trivial monodromy over a non-orientable surface of Euler characteristic less than $-1$.
\item
The non-trivial circle bundle with trivial monodromy over the projective plane.
\end{enumerate}
\end{prop}

He also asks:
\begin{quest}
Which (closed) $3$-manifolds have fundamental groups that are residually torsion-free nilpotent?
\end{quest}

The point of this section is to explore this question.  In the introduction, we claimed that if $\psi\in\Mod(\Sigma)$ and acts unipotently on $H_1(\Sigma,\bZ)$, then $\pi_1(M_{\psi})$ is residually torsion-free nilpotent.

\begin{proof}[Proof of Theorem \ref{t:rtfn}]
In light of the theory we have developed so far in this paper, the claim is almost obvious.  Let $H$ denote $\pi_1(\Sigma)$ and lift $\psi$ arbitrarily to $\Aut(H)$.  The argument in Lemma \ref{l:porder} implies that $\psi$ acts unipotently on $\gamma_i(H)/\gamma_{i+1}(H)$ for all $i$.

Since the action on $H_1(\Sigma)$ is unipotent, it follows that the semidirect product \[1\to H_1(\Sigma)\to N_1\to \bZ\to 1\] is nilpotent, where the conjugation action of $\bZ$ is given by $\psi$.  The kernel of the map $\pi_1(M_{\psi})\to N_1$ is precisely $\gamma_1(H)$.

By induction, we assume that the semidirect product \[1\to H/\gamma_i(H)\to N_i\to\bZ\to 1\] is a nilpotent group, where the conjugation action of $\bZ$ is again given by $\psi$.  The action of $\psi$ on $\gamma_i(H)/\gamma_{i+1}(H)$ is unipotent, so that the abelianization of the semidirect product \[1\to \gamma_i(H)/\gamma_{i+1}(H)\to Z_i\to\bZ\to 1\] has rank at least $2$.  Note also that conjugation within $H$ acts trivially on $\gamma_i(H)/\gamma_{i+1}(H)$.  It follows that there is a torsion-free quotient $A$ which fits into a central extension of the form \[1\to A\to N_i^A\to N_i\to 1.\]  Since the extension is central, it follows that $N_i^A$ is nilpotent.  Notice that $\psi$ again acts unipotently on the kernel of the map $\gamma_i(H)/\gamma_{i+1}(H)\to A$.  Since $\gamma_i(H)/\gamma_{i+1}(H)$ has finite rank, we eventually exhaust all of it, so that $N_{i+1}$ is nilpotent.

To complete the proof, we need only show that each $N_i$ is torsion free.  But this is obvious: we write $1\neq n\in N_i$ as $t^kh$, where $h\in H/\gamma_{i+1}(H)$.  We may obviously suppose that $k\neq 0$.  We choose $j$ so that $h\in\gamma_j(H)$ and $1\neq h\in\gamma_j(H)/\gamma_{j+1}(H)$.  So, if $h$ was nontrivial to start out, then we can write $n$ as $t^kh\in N_j$ for some $j\leq i$.  By the choice of $j$, $t$ acts trivially on $h$, so that $(t^kh)^m=t^{km}h^m$.  Since $\gamma_i(H)/\gamma_{i+1}(H)$ is torsion-free for all $i$ and since \[\bigcap_i\gamma_i(H)=\{1\},\] (cf. \cite{MKS}), we see that $\pi_1(M_{\psi})$ is residually torsion-free nilpotent.
\end{proof}

\begin{cor}
There exist finite volume hyperbolic $3$-manifolds with residually torsion-free nilpotent fundamental groups.
\end{cor}
By \cite{Wi}, the fundamental group of any finite volume hyperbolic $3$-manifold is not a limit group.  Since there are finite volume hyperbolic $3$-manifolds whose fundamental groups are residually $p$ at exactly one prime or even no primes, it follows that there are finite volume hyperbolic $3$-manifolds whose fundamental groups are not residually torsion-free nilpotent.

It follows easily from the argument in the proof of Theorem \ref{t:rtfn} that torus bundles have torsion-free nilpotent fundamental groups.  That the fundamental group of such a torus bundle is torsion-free also follows from Theorem \ref{t:rp1}, which shows that any such torus bundle is residually $p$ for every prime.  A nilpotent group can be residually $p$ if and only if its torsion subgroup is a $p$-group.
\begin{prop}
A hyperbolic torus bundle over the circle is not residually torsion-free nilpotent.
\end{prop}
\begin{proof}
We have argued through Theorem \ref{t:rp1} that hyperbolic torus bundles are never residually $p$ for all primes.  Alternatively, let $G$ be the fundamental group of a hyperbolic torus bundle.  Then $G^{ab}$ has rank one.  Furthermore, the monodromy acts irreducibly on $H_1(S^1\times S^1,\bQ)$, and it is easy to check that $[G,\pi_1(S^1\times S^1)]$ has finite index in $\pi_1(S^1\times S^1)$.  In particular, there can be no further torsion-free nilpotent quotient of $G$.
\end{proof}

Suppose that $M_{\psi}$ is a fibered $3$-manifold with fundamental group $G$ and let $N$ be a torsion-free nilpotent quotient of $G$.  Let $\phi:G\to N$ be the quotient map.  If $G'<G$ is a finite index subgroup, we may restrict $\phi$ to $G'$ to obtain a finite index subgroup $N'<N$ which is again nilpotent and torsion-free.  It follows that being residually torsion-free nilpotent is invariant under taking finite index subgroups.  We have given examples of mapping classes in the Torelli group (pure braids, the Torelli group of a multiply punctured disk) which lift to a finite cover of the base surface and act non-unipotently (cf. examples in \cite{K}).  The suspensions of such mapping classes have residually torsion-free nilpotent fundamental groups, so that there is no na\"ive converse to Theorem \ref{t:rtfn}.  Another complicating factor is that when $\psi$ is a Torelli mapping class, then $b_1(M_{\psi})=b_1(\Sigma)+1$, so that $M_{\psi}$ will admit infinitely many inequivalent fibrations (see \cite{T3}).  Furthermore, if such a fibration has fiber $\Sigma'$ with $b_1(\Sigma')>b_1(\Sigma)$, then the monodromy of that fibration cannot be in the Torelli group, and it seems unlikely that one would have so much control as to have all the fibrations to have unipotent homological monodromy.

To further illustrate the lack of a na\"ive converse to Theorem \ref{t:rtfn}, we give an alternative proof which uses Theorem \ref{t:rp1}.  Recall that if $\psi$ acts unipotently on $H_1(\Sigma,\bZ)$, then $\pi_1(M_{\psi})$ is residually torsion-free nilpotent.  If $\psi$ acts unipotently on $H_1(\Sigma,\bZ/p\bZ)$ for infinitely many primes, then $\pi_1(M_{\psi})$ is again residually torsion-free nilpotent.  One would like some sort of a converse to this fact.  To formulate a converse, suppose that a group $G$ is residually $p$ at an infinite set of primes $\mathcal{P}$ and let $g\in G$.  We say $G$ has {\bf property $U(\mathcal{P})$} if for each $g\in G$ there is a universal $m=m(G,\mathcal{P})$ such that $g$ survives in a $p$-group quotient of $G$ which has nilpotence class no larger than $m$, universally over $\mathcal{P}$ (recall that the nilpotence class of a group $G$ is the smallest $m$ such that $\gamma_m(G)$ is trivial).

\begin{prop}\label{p:tfn}
Let $G$ be a finitely generated group.  The following are equivalent:
\begin{enumerate}
\item
$G$ is residually torsion-free nilpotent.
\item
$G$ is residually $p$ at every prime and has property $U$ for the set of all primes.
\item
$G$ is residually $p$ at an infinite set of primes $\mathcal{P}$ and has property $U(\mathcal{P})$.
\end{enumerate}
\end{prop}
\begin{proof}
The first forward implication follows from the general theory of finitely generated torsion-free nilpotent groups which is due to Mal'cev (see \cite{R} for a nice exposition).  A torsion-free finitely generated nilpotent group sits as the group of integral points of a finite-dimensional real nilpotent Lie group.  The integer lattice can be reduced modulo any prime.  It is obvious that if $g$ survives in $G/\gamma_m(G)$ and this group is torsion-free, then $g$ survives in a $p$-group of nilpotence class no larger than $m$.  The second forward implication is trivial.

Let $\mathcal{P}$ be the infinite set of primes for which $G$ is residually $p$.  Since $\mathcal{P}\neq\emptyset$, $G$ is residually nilpotent, or equivalently \[\bigcap_i\gamma_i(G)=\{1\}.\]  Notice that $G^{ab}\otimes\bQ$ is nontrivial.  Indeed, otherwise $G^{ab}$ would have to be finite of order $n$.  Let $p\in\mathcal{P}$ be relatively prime to $n$.  Since $G$ is residually $p$, $G$ admits a quotient isomorphic to $\bZ/p\bZ$, and the quotient map factors through $G^{ab}$, which is a contradiction.  Notice that the property of being residually $p$ is invariant under passing to a subgroup.  In particular, merely the assumption that $\mathcal{P}$ is infinite guarantees that $G$ has many torsion-free quotients.

Furthermore, it is obvious that $G$ has no torsion.  For every $p\in \mathcal{P}$, there is a $p$-group quotient of $G$ where $g$ is nontrivial, and we may suppose that there is a universal $m=m(G)$ such that the smallest (in terms of nilpotence class) $p$-group quotient in which $g$ survives has nilpotence class $m$.  Thus we may find such a $p$-group quotient $Q_{p,g}$ in which $g$ survives.  We consider the product map \[G\to P=\prod_{p\in\mathcal{P}} Q_{p,g}.\]  We have that $P$ is nilpotent since each $Q_{p,g}$ has nilpotence class bounded by $m$.  It follows that in $G/\gamma_m(G)$, the image of $g$ will not be in the torsion subgroup $T(G/\gamma_m(G))$ (cf. \cite{R}), so that $g$ survives in a torsion-free nilpotent quotient of $G$.
\end{proof}

It seems unlikely that the condition $U$ on $G$ can be removed (cf. examples in Section \ref{s:odd}).  Notice that if $\psi$ is a mapping class which acts unipotently on $H_1(\Sigma,\bZ)$, then it is easy to show that $\pi_1(M_{\psi})$ has property $U$ with respect to the set of all primes.  As a corollary, we obtain another proof of Theorem \ref{t:rtfn}.

Wilton's result shows that orientable geometric $3$-manifolds whose fundamental groups are nontrivial limit groups must admit geometric structures modeled on $S^2\times\bR$, $\bR^3$ or $\bH^2\times\bR$.  As for nontrivial residually torsion-free nilpotent $3$-manifold groups, we have additionally exhibited examples with nil and $\bH^3$ geometry.  $S^3$ geometry is ruled out by nontriviality, and we have ruled out sol geometry.  We finally consider $\widetilde{PSL_2(\bR)}$ geometry:

\begin{thm}\label{t:psl2r}
Let $S^1\to M\to\Sigma$ be a circle bundle over a closed orientable surface $\Sigma$ with trivial monodromy and nontrivial Euler class $e\in H^2(\Sigma,\bZ)$.  Then $G=\pi_1(M)$ is residually torsion-free nilpotent.
\end{thm}
\begin{proof}
We have that $G$ fits into a non-split central extension \[1\to\bZ\to G\to\pi_1(\Sigma)\to 1.\]  Killing the central copy of $\bZ$, we see that for each $g\in G$ which projects nontrivially to $\pi_1(\Sigma)$, we have $1\neq g\in N$ for some $N\cong \pi_1(\Sigma)/\gamma_i(\pi_1(\Sigma))$.  These are all torsion-free, as is well-known from the work of Magnus.

Let $N$ be a torsion-free nilpotent quotient of $\pi_1(\Sigma)$.  There is a natural map $H^2(N,\bZ)\to H^2(\Sigma,\bZ)$ given by pullback.  If $e$ is in the image of the pullback, then there is a quotient $Q$ of $G$ which fits into a central extension \[1\to\bZ\to Q\to N\to 1.\]

Let $g$ be the genus of $\Sigma$, and fix a complex structure on $\Sigma$ together with a basepoint.  There is a canonical isomorphism $H_1(\Sigma,\bZ)\cong H_1(\bZ^{2g},\bZ)$ given by the period map.  By duality we obtain an isomorphism $H^1(\bZ^{2g},\bZ)\to H^1(\Sigma,\bZ)$.

There is an induced map $H^2(\bZ^{2g},\bZ)\to H^2(\Sigma,\bZ)$.  We have that $H_2(\bZ^{2g},\bZ)$ is naturally identified with $\Lambda^2H_1(\Sigma,\bZ)$ (cf. \cite{GH}, for instance).  The map $H^2(\bZ^{2g},\bZ)\to H^2(\Sigma,\bZ)$ is given by taking $\al,\beta\in H_1(\Sigma,\bZ)$, taking their Poincar\'e duals $\al^*$ and $\beta^*$, taking $\al\wedge\beta\in H_2(\bZ^{2g},\bZ)$ and sending it to $\al^*\cup\beta^*$, which gives a linear function $H_2(\bZ^{2g},\bZ)\to \bZ$ (hence a second cohomology class) and produces an element of $H^2(\Sigma,\bZ)$.  In particular, the map $H^2(\bZ^{2g},\bZ)\to H^2(\Sigma,\bZ)$ is surjective.

It follows that if $z$ is a nontrivial element of the central copy of $\bZ$ then $z$ is nontrivial in a torsion-free nilpotent quotient $Q$ of $G$ which can be described via \[1\to\bZ\to Q\to\bZ^{2g}\to 1.\]  In particular, each element of $G$ survives in a torsion-free nilpotent quotient of $G$.
\end{proof}

\begin{cor}
If $G$ is as in Theorem \ref{t:psl2r} and $p$ is a prime, then $G$ is residually $p$.
\end{cor}

Though the proof of Theorem \ref{t:psl2r} might seem a little ad hoc, there are good reasons why the quotient of $\pi_1(\Sigma)$ we consider is the abelianization.  Indeed, let $C=[\pi_1(\Sigma),\pi_1(\Sigma)]$ and let $N$ be a torsion-free nilpotent quotient of $\pi_1(\Sigma)$.  Suppose we have a torsion-free nilpotent quotient $N'$ of $G<PSL_2(\bR)$ which fits into a central of the form \[1\to\bZ\to N'\to N\to 1,\] where $G$ is a central extension of $\pi_1(\Sigma)$.  Since $C$ is a subgroup of $\pi_1(\Sigma)$, there is also a central extension of the form \[1\to \bZ\to G'\to C\to 1.\]  Since this central extension is classified by a second cohomology class of $C$, the extension splits.  Indeed, $C$ is free and hence has cohomological dimension zero (alternatively $C$ is a free and hence projective object in the category of groups so that any such extension must split).  It is therefore not surprising that the abelianization map on $\pi_1(\Sigma)$ would be sufficient for witnessing the residual torsion-free nilpotence of $G$.

In Section \ref{s:odd} we will briefly consider incompatibilities between residually $p$, residually torsion-free nilpotent, and RFRS groups.

\section{Some remarks on odd embeddings of groups}\label{s:odd}
We have seen above that if we are given a presentation of a fibered $3$-manifold as a surface together with a mapping class, one might not be able to tell easily from the action of the mapping class on the homology of the fiber whether or not the fundamental group of the fiber injects into its pro-$p$ completion.  This is a reflection of the somewhat flexible nature of profinite and pro-$p$ completions.

\begin{prop}
There exists an exhausting sequence of nested subgroups $\{K_i\}<\bZ^{\infty}$ such that $\bZ^{\infty}/K_i$ is a cyclic $p$-group for all $i$.
\end{prop}
\begin{proof}
Notice that $\bZ_p$, the ring of $p$-adic integers, satisfies $\bZ_p/p^k\bZ_p\cong\bZ/p^k\bZ$.  Furthermore, $\bZ_p$ is torsion-free and uncountable.  This means that a generic choice of countably many elements of $\bZ_p$ generate a copy of $\bZ^{\infty}$ in $\bZ_p$, whence the claim.
\end{proof}

More bizarre phenomena can occur with nonabelian groups.  Let $F_3=\langle x,y,z\rangle$, and enumerate the elements of $[F_3,F_3]$.  Let $Q_{m,n}$ be the quotient obtained by introducing the relations $c_n=z$, where $c_n$ is the $n^{th}$ element of $[F_3,F_3]$, and setting all $m$-fold commutators to be trivial.  Let $K_{m,n}$ be the kernel of the surjection $F_3\to Q_{m,n}$.
\begin{prop}
With the setup above, \[\bigcap_{m,n}K_{m,n}=\{1\}.\]
\end{prop}
\begin{proof}
Let $1\neq w\in F_3$, and write $w$ as a reduced word in the generators.  If there is no occurrence of $z$ or $z^{-1}$ in $w$, then we can view $w$ as a word in $F_2$.  Setting $c_n$ to be the identity and taking $m$ sufficiently large, we find that for some $m,n$, $w\notin K_{m,n}$.

Let $K$ be the kernel of the map $F_3\to F_2$ which kills the letter $z$.  By the previous paragraph, we may assume $w\notin K$.  A free basis for $K$ is given by conjugates of $z$ by words in $x$ and $y$.  Write $w$ with respect to this basis.  There is an $i$ such that every nonidentity subword of $w$ (in $x,y,z$) is nontrivial and such that they are all distinct in $F_3/\gamma_i(F_3)$.

Clearly $\gamma_i(F_3)$ is free and is equipped with a conjugation action by $F_3$.  Nontrivial elements of $F_3/\gamma_i(F_3)$ can be used to take a free basis element of $\gamma_i(F_3)$ and to produce other ones.  Choose such a basis element $\zeta$, and consider the quotient of $F_3$ given by imposing $z=\zeta$.  It follows that if $i$ is chosen to be sufficiently large, then the word given by substituting $\zeta$ for $z$ in $w$ will be nontrivial.  It follows that $w$ is nontrivial in some nilpotent quotient $N$ of $F_3$ which satisfies $1=\overline{z}\in N^{ab}$.
\end{proof}

Finally, we consider some of the interplay between residually torsion-free nilpotent $3$-manifold groups and the RFRS condition.  Recall that in \cite{A2}, Agol develops the RFRS condition.  A group is called {\bf RFRS} or {\bf residually finite rationally solvable} if there exists an exhaustive filtration of normal (in $G$) finite index subgroups of $G$ \[G=G_0>G_1>G_2>\cdots\] such that $G_{i+1}>ker\{G_i\to G_i^{ab}\otimes\bQ\}$.  Agol proves:
\begin{thm}
Let $M$ be a $3$-manifold with $\chi(M)=0$, and suppose that $\pi_1(M)$ is RFRS.  Let $\phi\in H^1(M,\bZ)$ be a non-fibered cohomology class.  Then there is a finite cover $M'\to M$ such that the pullback of $\phi$ to $H^1(M',\bZ)$ such that $\phi$ lies in the cone over the boundary of a fibered face of $H^1(M',\bZ)$.
\end{thm}

If $\gam$ is the fundamental group of a nontrivial circle bundle $M$ over a surface $\Sigma$, then $M$ should not fiber over the circle.  The nontriviality of the Euler class of the bundle allows us to verify this fact directly using Agol's theorem.  Thus we will make precise the statement that residual torsion-free nilpotence and RFRS have little to do with each other.

To warm up, we have:
\begin{prop}
The integral Heisenberg group is not RFRS.
\end{prop}
\begin{proof}
Let $\{G_i\}$ be any candidate for an exhausting sequence which witnesses the RFRS condition.  We present the integral Heisenberg group as \[H=\langle x,y,z\mid [x,y]=z,\,[x,z]=[y,z]=1\rangle.\]  We have the abelianization map $\phi:H\to\bZ^2$.  Since each $G_i<H$ has finite index, we may restrict the map $\phi$ to $G_i$ to obtain a finite index subgroup of $\bZ^2$.  Note that in general $\phi$ does not restrict to the abelianization map of $G_i$.  Let $a,b\in\bZ^2$ be two elements in the image of $G_i$ in $\bZ^2$ which pull back to $x^nz^i$ and $y^mz^j$ respectively.  It is easy to verify that the commutator of these two elements is $z^{\pm mn}$, depending on the order of the commutation.  It follows that $[G_i,G_i]$ contains a finite index subgroup of the center of $H$.  The kernel of the abelianization map of $H$ is precisely the center $Z$ of $H$.  By induction, we have that $Z<G_i$.  Since the abelianization map of $G_i$ sends $z$ to a finite order element, $Z$ is in the kernel of the map $G_i\to G_i^{ab}\otimes\bQ$, so that $Z<G_{i+1}$.  It follows that $\{G_i\}$ does not exhaust $H$.
\end{proof}
Though it is not clear from the definition of RFRS, the same argument shows that $H$ is not virtually RFRS.  It is easy to see from this argument that a torsion-free nilpotent group $N$ which is classified by a nontrivial $e\in H^2(N/Z,Z)$ will generally not be RFRS.

\begin{prop}
Let $\gam=\pi_1(M)$ where $M$ is an orientable circle bundle with trivial monodromy over a closed orientable surface $\Sigma$, and suppose the bundle has a nontrivial Euler class $e\in H^2(\Sigma,\bZ)$.  Then $\gam$ is not RFRS.
\end{prop}
\begin{proof}
Let $\{G_i\}$ be any candidate for an exhausting sequence which witnesses the RFRS condition on $\gam$.  We will again show that $\{G_i\}$ cannot exhaust $\gam$.  We view $\gam$ as a central extension \[1\to\bZ\to\gam\to\pi_1(\Sigma)\to 1.\]  Consider $\gam^{ab}$.  Clearly there is a surjection $\gam^{ab}\to H_1(\Sigma,\bZ)$.  Suppose that the central copy of $\bZ$ in $\gam$ maps injectively to $\gam^{ab}$.  We would then be able to describe $\gam^{ab}$ as a central extension of the form \[1\to\bZ\to \gam^{ab}\to H_1(\Sigma,\bZ)\to 1,\] and since $\gam^{ab}$ is abelian the extension would have to split.  In particular, its classifying cocycle in $H^2(\bZ^{2g},\bZ)$ would be trivial, where $g$ is the genus of $\Sigma$.  The pullback cocycle in $H^2(\Sigma,\bZ)$ would also be zero, so that the extension describing $\gam$ itself would have to split.  It follows that the central copy of $\bZ$ is in the kernel of the map $\gam\to\gam^{ab}\otimes\bQ$.  We therefore have $\bZ<G_1$.

We inductively suppose that $\bZ<G_i$.  Since $G_i$ has finite index, the image of the restriction of the map $\gam\to H_1(\Sigma,\bZ)$ to $G_i$ has finite index, and we call this subgroup $A$.  Then $A$ gives rise to a central extension of the form \[1\to\bZ_i\to N\to A\to 1,\] where the central copy of $\bZ=\bZ_i$ is $\bZ\cap G_i$.  The inclusion $A<\bZ^{2g}$ gives a map $H^2(\bZ^{2g},\bZ)\to H^2(A,\bZ)$ (the restriction map).  Recall that by general cohomology of groups there is a natural corestriction map $H^2(A,\bZ)\to H^2(\bZ^{2g},\bZ)$, and the composition of the corestriction with the restriction (often written $Cor\circ Res$) is multiplication by $[\bZ^{2g}:S]$ (cf. \cite{Br}).  It follows that if $e\neq 0$ then $Cor\circ Res(e)\neq 0$ so that $Res(e)\neq 0$.  In particular, the extension \[1\to\bZ_i\to N\to A\to 1\] is non-split.  Since $A$ is a torsion-free quotient of $G_i^{ab}$, we cannot have $\bZ_i$ mapping injectively to $G_i^{ab}$, so that $\bZ_i$ is in the kernel of the map $G_i\to G_i^{ab}\otimes\bQ$.  It follows that $\bZ_i<G_{i+1}$.  Since we assumed $\bZ_i=\bZ$, we have that the original central copy of $\bZ$ was contained in each $G_i$, which is what we set out to prove.
\end{proof}

Since Agol's theorem is a criterion which is useful for proving that many hyperbolic $3$-manifolds are virtually fibered, the previous propositions shows that it is unlikely that residual torsion-free nilpotence of hyperbolic $3$-manifold groups will be helpful in resolving questions about virtual fibering, and Proposition \ref{p:tfn} shows that being residually $p$ at many primes is also unlikely to help.

\end{document}